\documentclass[11pt,reqno]{amsart}
\usepackage{xcolor}

\usepackage[utf8]{inputenc}
\usepackage[T1]{fontenc}
\usepackage{amsmath,amssymb,amsthm,mathtools}
\usepackage[shortlabels]{enumitem}
\usepackage[margin=1.05in]{geometry}
\usepackage{tikz}
\usepackage[colorlinks=true,linkcolor=blue,citecolor=blue,urlcolor=blue]{hyperref}

\allowdisplaybreaks
\numberwithin{equation}{section}

\newtheorem{theorem}{Theorem}[section]
\newtheorem{proposition}[theorem]{Proposition}
\newtheorem{lemma}[theorem]{Lemma}
\newtheorem{corollary}[theorem]{Corollary}
\theoremstyle{definition}
\newtheorem{definition}[theorem]{Definition}
\newtheorem{problem}[theorem]{Problem}
\theoremstyle{remark}
\newtheorem{remark}[theorem]{Remark}
\newtheorem{example}[theorem]{Example}

\newcommand{\R}{\mathbb R}
\newcommand{\one}{\mathbf 1}
\newcommand{\ip}[2]{\left\langle #1,#2\right\rangle}
\newcommand{\norm}[1]{\left\lVert #1\right\rVert}
\newcommand{\cE}{\mathcal E}
\newcommand{\cL}{\mathcal L}

\title[Inverse-weight Ricci flow on trees]
{Curvature Diffusion of Inverse-weight Lin--Lu--Yau Ricci Flow on Finite Trees} 
\author{Shuliang Bai}
\address{Beijing Institute of Mathematical Sciences and Applications,
Beijing, China}
\email{baishuliang@bimsa.cn}
\author{Bobo Hua}
\address{School of Mathematical Sciences, LMNS, Fudan University,
Shanghai 200433, China}
\email{bobohua@fudan.edu.cn}
\date{}

\begin{document}

\begin{abstract}
We study the inverse-weight Lin–Lu–Yau Ricci flow on finite trees, and prove the convergence of the curvature on every edge via its diffusion equation. Moreover, the limiting curvature is the
unique minimum-norm point of a polyhedron determined by the tree. We characterize possible
limits of the normalized edge weights, give criteria for their convergence, and show that different
positive initial metrics have equivalent omega-limit sets.
\end{abstract}

\maketitle

\section{Introduction}

Ollivier curvature compares the one-step probability measures based at nearby
points by optimal transport \cite{Ollivier2009}.  Lin, Lu, and Yau introduced its
idleness derivative for graphs \cite{LinLuYau2011}; the dependence of graph
Ollivier curvature on the idleness parameter was subsequently analyzed in detail
in \cite{BourneEtAl2018}.  The continuous Ricci flow on a
weighted graph was introduced in previous work \cite{BaiEtAl2024}.  In
\cite{BaiHuaLinLiu2025}, we studied the Ricci flow in the inverse-weight model on finite trees: for each edge $e$ with edge weight $w_e,$ the edge length is given by $w_e$ and the transition rate of simple random walk on $e$ is proportional to  $w_e^{-1},$ which is related to electric network model in physics and Markov chain in probability. Some normalized edge weights can tend to zero along that flow.
The limiting statement in that paper is therefore formulated on the closed simplex
and gives constant curvature on the positive support.  Recall that a caterpillar is a
tree whose non-leaf vertices form a path.  For a flat caterpillar limit, all
normalized leaf weights tend to zero.  Internal edges may also tend to zero and
separate the positive support into several components.

In the inverse-weight model, both the transition kernel and the path metric depend on the evolving edge
weights.  In the
fixed-kernel model of \cite{BaiLiuLai2026}, the transition kernel is unchanged
throughout the evolution and only the path metric varies; on a tree the resulting
flow is a linear system.  In the uniform-neighbor model of
\cite{BaiChengHua2026}, the measure at $x$ distributes its non-idle mass uniformly
among the $d_x$ neighbors and is independent of the edge lengths, which enter only
through the path metric.  In that setting the Einstein equation is equivalent to a
Perron eigenvalue problem for the Ricci matrix.  By contrast, the inverse-weight
flow studied here is nonlinear even on trees.

There are several convergence results for other curvature flows on general weighted
graphs.  Ma and Yang considered piecewise-linear curvature flows with surgery
\cite{MaYang2025}.  For graphs of girth at least six, Lin and Liu proved convergence
of a prescribed-curvature flow whose transition probabilities are proportional to
the edge weights \cite{LinLiu2026}; the corresponding prescribed-curvature Calabi
flow was studied in \cite{LiWangYuanZheng2026}.  These evolutions differ from the
metric-dependent inverse-weight flow considered here.  To our knowledge, an exact
curvature evolution equation of the form obtained below has not previously been
established for this inverse-weight flow, and its behavior on graphs with cycles
remains open.


Let $T=(V,E)$ be a finite tree, and assign a positive length $w_e$ to every
edge $e$.  These lengths determine the path metric $d_w$ on $V$.  At a vertex
$x$, put
\[
 D_x(w):=\sum_{e\ni x}\frac1{w_e}.
\]
For $\alpha\in[0,1)$, the lazy probability measure based at $x$ is given by
\[
 m_x^\alpha(x)=\alpha,
 \qquad
 m_x^\alpha(z)=(1-\alpha)\frac{w_{xz}^{-1}}{D_x(w)}
 \quad(z\sim x).
\]
This is the inverse-weight model, in which a shorter edge carries
a larger transition probability. In electric networks, the edge length $w_e$ is interpreted as the resistance of the edge, and hence $w_e^{-1}$ is its conductance.  If $e=xy$, its
Lin--Lu--Yau curvature is
\[
 \kappa_e(w)
 =\lim_{\alpha\uparrow1}
 \frac{1-W_1(m_x^\alpha,m_y^\alpha)/w_e}{1-\alpha},
\]
where $W_1$ is the Wasserstein distance associated with $d_w$.  The same edge
lengths therefore determine both the transportation cost and the transition
kernel. Thus changing an edge length affects both the transportation metric and the probability measures entering the curvature. The resulting feedback is nonlinear even on a tree.

We consider the unnormalized flow
\begin{align}\label{eq:unnormalized}
                  \dot{\widetilde w}_e
                  =-\kappa_e(\widetilde w)\widetilde w_e. 
\end{align}
For the normalized flow, we assume that
\[ w_e(0)>0\quad(e\in E),
 \qquad
 \sum_{e\in E}w_e(0)=1:\]
\begin{align}\label{eq:normalized}
 \dot w_e=-w_e\bigl(\kappa_e(w)-K(w)\bigr),
 \qquad
 K(w):=\sum_{f\in E}w_f\kappa_f(w).
\end{align}
The normalization $\sum_{e\in E}w_e(t)=1$ is preserved by
\eqref{eq:normalized}.
Curvature is unchanged when all edge lengths are multiplied by the same
positive constant.  Consequently, the normalized flow is obtained from the
unnormalized one by dividing by the total edge length.  The two flows have the
same curvatures and the same edge-length ratios; normalization merely removes
their common scale.

\subsection*{Main results}

Let
$\kappa=(\kappa_e)_{e\in E}$ be the curvature vector, and write $m=|E|$.  We use
$\dot f=df/dt$ for differentiation with respect to time.  The weighted Laplacian
$\cL(w)$ used below is defined on the edge set: two edges are
adjacent if they meet at a vertex of degree at least three, and its conductances are
given by \eqref{eq:conductance}.

\begin{theorem}\label{thm:heat}
On a finite tree, along either the unnormalized flow \eqref{eq:unnormalized} or the normalized \eqref{eq:normalized},
\begin{equation}\label{eq:heat}
                 \dot\kappa=-\cL(w)\kappa.
\end{equation}
For each positive metric, $\cL(w)$ is symmetric and positive semidefinite,
$\cL(w)\one=0$, and
\begin{equation}\label{eq:quadratic}
 \ip{z}{\cL(w)z}
 =\sum_{\{e,f\}:a_{ef}>0}a_{ef}(z_e-z_f)^2\ge0.
\end{equation}
\end{theorem}

A related heat equation for Ollivier curvature and general graph Laplacians was studied in \cite{MunchWojciechowski2019}.  Here the Laplacian evolves with the metric, and no uniform lower bound for its positive conductances is assumed.  Nevertheless, every edge curvature converges, and the limits determine the exponential rates of the edge weights.

\begin{theorem}\label{thm:kappalimit}
For every $e\in E$, the limit
\begin{equation*}
                    c_e:=\lim_{t\to\infty}\kappa_e(t)
\end{equation*}
exists.  Put
\begin{equation*}
       c_*:=\min_{e\in E}c_e,
       \qquad E_*:=\{e\in E:c_e=c_*\}.
\end{equation*}
If $\widetilde w$ denotes the unnormalized solution and $w$ its normalization, then
\begin{equation}\label{eq:exponentialrates}
 \lim_{t\to\infty}\frac1t\log\widetilde w_e(t)=-c_e,
 \qquad
 \lim_{t\to\infty}\frac1t\log w_e(t)=c_*-c_e.
\end{equation}
Every edge outside $E_*$ therefore has exponentially vanishing
normalized weight.
\end{theorem}

If $w^*$ is an omega-limit point of the normalized trajectory, put
\[
 E^+(w^*):=\{e\in E:w_e^*>0\}.
\]
Theorem~\ref{thm:kappalimit} gives
\begin{equation}\label{eq:supportminimal}
 E^+(w^*)\subseteq E_*,
 \qquad
 \lim_{t\to\infty}\kappa_e(t)=c_*
 \quad(e\in E^+(w^*)),
\end{equation}
and
\[
 \sum_{e\notin E_*}w_e(t)\longrightarrow0.
\]
Thus any nonconvergent normalized edge weight must belong to $E_*$.
If $E_*=\{e_*\}$, then $w_{e_*}(t)\to1$ and all other normalized
weights tend to zero.

For a vertex $x$, let $E(x)$ be the set of edges incident to $x$.  For
$A\subseteq E$, set
\begin{equation*}
 \beta(A):=\sum_{\substack{x\in V\\E(x)\subseteq A}}(2-d_x),
 \qquad
 \lambda(A):=\frac{\beta(A)}{|A|}\quad(A\ne\varnothing).
\end{equation*}
For $z\in\R^E$, write $z(A)=\sum_{e\in A}z_e$, and define
\begin{equation*}
 B(\beta):=
 \left\{z\in\R^E:z(A)\le\beta(A)\text{ for every }A\subseteq E,
                         \ z(E)=\beta(E)\right\}.
\end{equation*}
The next result identifies the limiting curvature with a canonical
point of $B(\beta)$.

\begin{theorem}\label{thm:combinatorial}
The function $\beta$ is submodular, and the limiting curvature vector
$c=(c_e)_{e\in E}$ is the unique solution of
\begin{equation}\label{eq:minimumnormbase}
             \min_{z\in B(\beta)}\frac12\sum_{e\in E}z_e^2.
\end{equation}

Let $\alpha_1<\cdots<\alpha_r$ be the distinct entries of $c$, and put
\[
 A_0=\varnothing,
 \qquad A_j:=\{e\in E:c_e\le\alpha_j\}\quad(1\le j\le r).
\]
Then $A_r=E$ and
\begin{equation}\label{eq:tightlevels}
                         \beta(A_j)=\sum_{e\in A_j}c_e.
\end{equation}
Moreover, for $0\le j<r$,
\begin{equation}\label{eq:principalchain}
 \alpha_{j+1}
 =\min_{A_j\subsetneq A\subseteq E}
   \frac{\beta(A)-\beta(A_j)}{|A|-|A_j|},
\end{equation}
and $A_{j+1}$ is the unique largest set attaining this minimum.Thus \eqref{eq:principalchain} characterizes the limiting curvature
vector recursively in terms of the tree, without reference to the
initial metric.
In particular,
\begin{equation}\label{eq:cstarcombinatorial}
 c_*=\min_{\varnothing\ne A\subseteq E}\lambda(A).
\end{equation}
The set $E_*$ is the unique largest set attaining this minimum; equivalently,
\begin{equation}\label{eq:Estarunion}
 E_*=\bigcup\{A\subseteq E:A\ne\varnothing,
                         \ \beta(A)=c_*|A|\}.
\end{equation}
In particular, $\beta(E_*)=c_*|E_*|$.  The sets attaining the minimum are
closed under unions and under nonempty intersections.
\end{theorem}

It follows that $c$, $c_*$, and $E_*$ depend only on the tree and not
on the positive initial metric.  Minimum-norm points of base polyhedra
and their tight level sets are classical objects in submodular
optimization; see \cite{Fujishige1980,Fujishige2005}.

Curvature convergence alone does not imply convergence of the
normalized weights.  All edges in $E_*$ have the same exponential
rate, while their relative weights may still vary.  We therefore
describe the equations satisfied by omega-limit points.

For a nonempty set $C\subseteq E$, put
\[
 \Delta_C:=\left\{p\in\R^C:p_e>0\text{ for every }e\in C,
                         \ \sum_{e\in C}p_e=1\right\}.
\]
We identify $p\in\Delta_C$ with the vector in $\R^E$ obtained by setting
$p_e=0$ for $e\notin C$.

Suppose that $w(t_j)\to p$ as $t_j\to\infty,$ and that $C$ is the positive support of $p$.  If
$e\in C$ and some edge outside $C$ is incident to an endpoint $x$ of $e$, then
$D_x(t_j)\to\infty$ and the contribution of $x$ to $\kappa_e(t_j)$ tends to zero.
In the limit, only vertices $x$ for which $E(x)\subseteq C$ contribute to the
curvature.  For $p\in\Delta_C$ and $e\in C$, define
\begin{equation}\label{eq:boundarycurvature}
 \kappa_e^C(p):=
 \sum_{\substack{x\in e\\E(x)\subseteq C}}
 \frac{2-d_x}{p_e\displaystyle\sum_{f\in E(x)}p_f^{-1}},
\end{equation}
and for any $c\in \R$ and $C\subseteq E,$ write the set of Einstein metrics of curvature $c$ supported on $C$ as
\begin{equation*}
 \mathcal S_C(c):=
 \left\{p\in\Delta_C:\kappa_e^C(p)=c\text{ for every }e\in C\right\}.
\end{equation*}

These equations depend only on the tree and the support $C$.
Pathwise consequences were obtained in
\cite[Section~6]{BaiHuaLinLiu2025}; here we use the full system on the
positive support.

\begin{theorem}\label{thm:algebraiccriterion}
Let $w^*$ be an omega-limit point of the normalized trajectory and set
$C=E^+(w^*)$.  Then
\[
                         C\subseteq E_*,
 \qquad                  w^*|_C\in\mathcal S_C(c_*).
\]
Consequently, if
\begin{equation}\label{eq:candidates}
 \bigcup_{\varnothing\ne C\subseteq E_*}\mathcal S_C(c_*)
\end{equation}
is finite, then the normalized metric converges.  In particular, it is enough that
$\mathcal S_C(c_*)$ contain at most one point for every nonempty
$C\subseteq E_*$.
\end{theorem}

To describe the solution set $\mathcal S_C(c)$, let $G_C$ be the
graph whose vertices are the edges in $C$.  Two edges are adjacent in
$G_C$ if they meet at a vertex $x$ with $d_x\ge3$ and
$E(x)\subseteq C$.

\begin{theorem}\label{thm:boundarystructure}
Let $C\subseteq E$ be nonempty.  Every $p\in\mathcal S_C(c)$ satisfies
\begin{equation}\label{eq:balanceidentity}
 |C|c=\sum_{\substack{x\in V\\E(x)\subseteq C}}(2-d_x).
\end{equation}
In particular, the support $C$ itself determines the only possible value of $c$.

Suppose that $G_C$ has connected components $C_1,\ldots,C_s$ and that
$\mathcal S_C(c)$ is nonempty.  Then there are uniquely determined vectors
$q^{(i)}\in\Delta_{C_i}$ such that
\begin{equation}\label{eq:SCsimplex}
 \mathcal S_C(c)=
 \left\{\sum_{i=1}^s\theta_iq^{(i)}:
       \theta_i>0,\ \sum_{i=1}^s\theta_i=1\right\}.
\end{equation}
When $G_C$ is connected, $\mathcal S_C(c)$ is a single point; otherwise it is an
open simplex of dimension $s-1$.

Let $w^*$ be an omega-limit point whose positive support $C$ has maximal cardinality
among the positive supports of all omega-limit points.  If $G_C$ is connected, then
$w(t)\to w^*$.  In particular, this conclusion holds if
$E^+(w^*)=E_*$ and $G_{E_*}$ is connected.
\end{theorem}

\begin{theorem}\label{thm:positivecriterion}
Let $C\subseteq E$ be nonempty, suppose that $G_C$ is connected, and put
$c=\lambda(C)$.  Then
\begin{equation}\label{eq:strictcriterion}
 \mathcal S_C(c)\ne\varnothing
 \quad\Longleftrightarrow\quad
 \lambda(A)>c
 \quad\text{for every }\varnothing\ne A\subsetneq C.
\end{equation}
When these conditions hold, $\mathcal S_C(c)$ consists of one point.

Consequently, if $w^*$ is an omega-limit point, $C=E^+(w^*)$, and
$C_1,\ldots,C_s$ are the connected components of $G_C$, then $C$ and each
$C_i$ attain the global minimum $c_*$ of $\lambda$ over the nonempty subsets
of $E$.  Moreover, each $C_i$ is inclusion-minimal among the minimizing sets.
\end{theorem}

\begin{corollary}\label{cor:Estarcriterion}
Let $C_1,\ldots,C_s$ be the connected components of $G_{E_*}$.  Then
\begin{equation}\label{eq:componentbalance}
                         \beta(C_i)=c_*|C_i|
                         \qquad(1\le i\le s),
\end{equation}
and
\begin{equation}\label{eq:componentcriterion}
 \mathcal S_{E_*}(c_*)\ne\varnothing
 \quad\Longleftrightarrow\quad
 \beta(A)>c_*|A|
 \quad\text{for every }i
 \text{ and every }\varnothing\ne A\subsetneq C_i.
\end{equation}
\end{corollary}

Corollary~\ref{cor:Estarcriterion} gives a combinatorial criterion for
the existence of a positive solution on $E_*$.  When such a solution
exists, Theorem~\ref{thm:boundarystructure} fixes the normalized
profile on each component of $G_{E_*}$, although the limiting equations
alone do not determine the relative masses of different components.
The flow nevertheless determines these masses and converges:

\begin{theorem}\label{thm:fullsupport}
Assume that
\[
                         \mathcal S_{E_*}(c_*)\ne\varnothing.
\]
Then the normalized metric converges to a limit $w^\infty$ whose positive support
is exactly $E_*$.  Moreover, $w^\infty|_{E_*}\in\mathcal S_{E_*}(c_*)$, and
\[
 \lim_{t\to\infty}\frac{w_e(t)}{w_f(t)}
 =\frac{w_e^\infty}{w_f^\infty}\in(0,\infty)
 \qquad(e,f\in E_*).
\]
No connectedness assumption on $G_{E_*}$ is required.  In particular, if one
 omega-limit point has positive support $E_*$, then the whole trajectory converges to
that omega-limit point
\end{theorem} See Example~\ref{ex:cuttingedge}.

For a non-caterpillar tree, the exponential growth of the total internal
unnormalized weight proved in \cite[Lemma~8]{BaiHuaLinLiu2025}, together with
\eqref{eq:exponentialrates}, gives $c_*<0$.  Every edge in $E_*$ therefore has negative
limiting curvature, and $E_*\ne E$: indeed,
$E_*=E$ would imply $c_*=2/|E|>0$ by Proposition~\ref{prop:GB}.  The case $E_*=E$
can occur only in the caterpillar class, although not every caterpillar has this
property.  Negativity of $c_*$ does not, however, imply
$\mathcal S_{E_*}(c_*)\ne\varnothing$.  Example~\ref{ex:nestedminimizer} gives a
non-caterpillar tree for which this set is empty.

The final result compares solutions arising from different positive
initial metrics on the same tree.

\begin{theorem}\label{thm:trajectorycomparison}
Let $\widetilde w$ and $\widetilde v$ be two unnormalized solutions with positive
initial metrics, and let $w$ and $v$ be their normalizations.  For every $e\in E$,
the limit
\[
 r_e:=\lim_{t\to\infty}\frac{\widetilde v_e(t)}{\widetilde w_e(t)}
\]
exists and belongs to $(0,\infty)$.  For $p\in\overline\Delta_E$, define
\[
 T_r(p)_e:=\frac{r_ep_e}{\sum_{f\in E}r_fp_f}.
\]
Then
\begin{equation*}
                    \norm{v(t)-T_r(w(t))}\longrightarrow0
\end{equation*}
and
\begin{equation}\label{eq:omegaprojective}
                    \omega(v)=T_r\bigl(\omega(w)\bigr).
\end{equation}
Here $\omega(w)$ and $\omega(v)$ denote the sets of  omega-limit points of the two
normalized trajectories.  The map $T_r$ is a homeomorphism of
$\overline\Delta_E$, with inverse $T_{r^{-1}}$, where
$(r^{-1})_e=r_e^{-1}$.  Consequently,
\[
 \{E^+(p):p\in\omega(w)\}=\{E^+(q):q\in\omega(v)\},
\]
and $w$ converges if and only if $v$ converges.  If $w(t)\to p$, then
$v(t)\to T_r(p)$.
\end{theorem}



Section~\ref{sec:setup} records the tree formula and Gauss--Bonnet identity.
Section~\ref{sec:heat} proves the curvature heat equation and the
convex-potential identity.  Section~\ref{sec:dissipation} establishes
curvature convergence and its combinatorial description.
Section~\ref{sec:weights} studies limiting normalized weights,
convergence criteria, and the comparison of trajectories with different
positive initial metrics.

\section{The inverse-weight model on a tree}\label{sec:setup}

Let $T=(V,E)$ be a finite tree with at least one edge, let $d_x$ denote the
combinatorial degree of $x$, and assign each edge $e$ a length $w_e>0$.  These lengths induce the path metric on
$V$.  We write $\one$ for the all-one vector indexed by $E$.  All inner products
and norms below are Euclidean.  The open and closed normalized simplices are
\begin{equation*}
 \Delta_E:=\left\{q=(q_e)_{e\in E}:q_e>0,\ \sum_{e\in E}q_e=1\right\},
 \qquad
 \overline\Delta_E:=\left\{q=(q_e)_{e\in E}:q_e\ge0,\ \sum_{e\in E}q_e=1\right\}.
\end{equation*}
For $x\in V$ put
\begin{equation*}
        D_x=D_x(w):=\sum_{e\ni x}\frac1{w_e}.
\end{equation*}
For $\alpha\in[0,1)$ define the lazy measure based at $x$ by
\[
 m_x^\alpha(x)=\alpha,
 \qquad
 m_x^\alpha(z)=(1-\alpha)\frac{w_{xz}^{-1}}{D_x}\quad(z\sim x).
\]
Let $W_1$ be the $1$-Wasserstein distance for the induced path metric.  If $e=xy$,
its Lin--Lu--Yau curvature is
\[
 \kappa_e=\lim_{\alpha\uparrow1}
 \frac{1-W_1(m_x^\alpha,m_y^\alpha)/w_e}{1-\alpha}.
\]
For the inverse-weight kernel, the following weighted tree formula was proved in
\cite{BaiEtAl2024,BaiHuaLinLiu2025}.

\begin{proposition}\label{prop:treeformula}
For every edge $e=xy$,
\begin{equation*}
  \kappa_e
  =\frac{2-d_x}{w_eD_x}+\frac{2-d_y}{w_eD_y}.
\end{equation*}
In particular, curvature is invariant under the common rescaling $w\mapsto sw$.
Moreover,
\[
 |\kappa_{xy}|\le |2-d_x|+|2-d_y|,
\]
so the curvatures are bounded by a constant depending only on the tree.
\end{proposition}

\begin{proposition}\label{prop:GB}
For every positive metric on a finite tree,
\begin{equation*}
             \sum_{e\in E}\kappa_e=2.
\end{equation*}
\end{proposition}

\begin{definition}\label{def:support}
For a normalized trajectory, let
\begin{equation*}
\omega(w):=\left\{w^*\in\overline\Delta_E:
 \text{there is a sequence }t_j\to\infty\text{ such that }w(t_j)\to w^*\right\}.
\end{equation*}

The elements of $\omega(w)$ are called the omega-limit points of the
normalized trajectory.
The convergence here is in $\R^E$; equivalently, along the same sequence
$(t_j)$ one has $w_e(t_j)\to w_e^*$ for every $e\in E$.  Since
$\overline\Delta_E$ is compact, $\omega(w)$ is nonempty.
For $w^*\in\omega(w)$, its positive support is
\[
 E^+(w^*):=\{e\in E:w_e^*>0\}.
\]
An internal edge $e$ with $w_e^*=0$ is called a \emph{cutting edge of $w^*$} when
each component of $T\setminus\{e\}$ contains an edge of $E^+(w^*)$.
\end{definition}

\begin{remark}\label{rem:boundarymetric}
The equality $w_e^*=0$ only gives $w_e(t_j)\to0$ along a sequence defining $w^*$;
the same edge may have positive weight at another  omega-limit point.  A cutting edge
of $w^*$ need not tend to zero along the full trajectory.  In the flat
caterpillar case of \cite{BaiHuaLinLiu2025}, every normalized leaf weight does tend
to zero.  Internal cutting edges may also occur at a boundary  omega-limit point and
split its positive support into several components.
\end{remark}

\section{Curvature diffusion and the convex potential}\label{sec:heat}

\subsection{The curvature heat equation}

Let $e\ne f$ share a vertex $x$.  Define
\begin{equation}\label{eq:conductance}
 a_{ef}^{x}(w):=\frac{d_x-2}{D_x^2w_ew_f}
 \quad\text{if }d_x\ge3,
 \qquad a_{ef}^{x}:=0\quad\text{if }d_x\le2.
\end{equation}
Two distinct edges of a tree share at most one vertex, so we write $a_{ef}=a_{ef}^x$
when they meet at $x$, and $a_{ef}=0$ otherwise.  Let $\cL(w)$ be the matrix
\begin{equation}\label{eq:L}
  (\cL z)_e=\sum_{f\ne e}a_{ef}(z_e-z_f).
\end{equation}
The matrix $\cL$ is the weighted Laplacian of the graph whose vertices are $E$ and whose
adjacencies are pairs of tree edges meeting at a vertex of degree at least three.  We
call this graph the \emph{conductive line graph}.

\begin{proof}[\normalfont\bfseries Proof of Theorem~\ref{thm:heat}]
Fix $e=xy$ and write its $x$-endpoint contribution as
\[
       h_{x,e}=\frac{2-d_x}{w_eD_x}.
\]
For the unnormalized flow,
\[
 \frac{\dot D_x}{D_x}
 =\frac1{D_x}\sum_{f\ni x}\frac{\kappa_f}{w_f}
 =:\bar\kappa_x.
\]
Direct differentiation gives
\[
 \dot h_{x,e}
 =\frac{d_x-2}{w_eD_x}(\bar\kappa_x-\kappa_e)
 =\sum_{\substack{f\ni x\\f\ne e}}
   \frac{d_x-2}{D_x^2w_ew_f}(\kappa_f-\kappa_e).
\]
When $d_x\le2$ the same formula is zero: for $d_x=2$ the coefficient vanishes,
whereas a leaf has no second incident edge.  Adding the two endpoint contributions
proves \eqref{eq:heat}.  For the normalized flow,
$\dot w_e/w_e=-\kappa_e+K$ and $\dot D_x/D_x=\bar\kappa_x-K$; the two $K$ terms cancel,
so the calculation is identical.  The symmetry $a_{ef}=a_{fe}$ follows from
\eqref{eq:conductance}, and \eqref{eq:L} shows that every row sum is zero.  Moreover,
\begin{align*}
 \ip{z}{\cL z}
 &=\sum_e\sum_{f\ne e}a_{ef}z_e(z_e-z_f)\\
 &=\frac12\sum_e\sum_{f\ne e}a_{ef}(z_e-z_f)^2\\
 &=\sum_{\{e,f\}:a_{ef}>0}a_{ef}(z_e-z_f)^2.
\end{align*}
In the second equality we have paired the terms indexed by $(e,f)$ and $(f,e)$.
This proves \eqref{eq:quadratic}.
\end{proof}

\begin{remark}\label{rem:insulators}
The support of $\cL(w)$ is independent of the positive metric, but its conductances
are not.  A degree-two vertex contributes no conductance to $\cL(w)$, although it
remains part of the path metric.  At a vertex of degree at least three, a collapsing
incident edge can make $D_x\to\infty$ and force some $a_{ef}\to0$.  The matrix $\cL(w)$
need not remain uniformly elliptic as $t\to\infty$.
\end{remark}

\subsection{The convex potential}

Write $y_e=\log w_e$ and define
\begin{equation*}
 \Phi(y):=\sum_{x\in V}(d_x-2)\log D_x(y),
 \qquad
 D_x(y):=\sum_{e\ni x}e^{-y_e}.
\end{equation*}

\begin{proposition}\label{prop:potential}
For every $e\in E$,
\begin{equation}\label{eq:gradPhi}
                    \frac{\partial\Phi}{\partial y_e}=\kappa_e,
 \qquad             \nabla^2\Phi(y)=\cL(w)\ge0.
\end{equation}
The function $\Phi$ is convex and the unnormalized logarithmic flow is
$\dot y=-\nabla\Phi$.  If $P=I-m^{-1}\one\one^T$, then along either the unnormalized flow \eqref{eq:unnormalized} or the normalized flow \eqref{eq:normalized},
\begin{equation}\label{eq:projectedgradient}
                     \frac{d}{dt}(Py)=-P\nabla\Phi(y).
\end{equation}
The kernel of $\nabla^2\Phi$ consists of the vectors $v$ such that $v_e=v_f$
whenever $e$ and $f$ meet at a vertex of degree at least three.  Equivalently,
$v$ is constant on each connected component of the conductive line graph.  In particular,
if the conductive line graph is connected, then $\Phi$ is strictly convex on
$\one^\perp$.
\end{proposition}

\begin{proof}
If $e\ni x$, then
\[
 \frac{\partial}{\partial y_e}(d_x-2)\log D_x
 =-(d_x-2)\frac{e^{-y_e}}{D_x}
 =\frac{2-d_x}{w_eD_x}.
\]
Summing over the two endpoints of $e$ gives the first identity in
\eqref{eq:gradPhi}.  For distinct edges $e,f$ meeting at $x$, the mixed derivative
is
\[
 -(d_x-2)\frac{e^{-y_e}e^{-y_f}}{D_x^2}=-a_{ef}.
\]
The diagonal entries are the corresponding row sums, which proves
$\nabla^2\Phi=\cL(w)$.  Theorem~\ref{thm:heat} now gives positive semidefiniteness.
The equations for the normalized and unnormalized flows differ in logarithmic
coordinates only by a multiple of $\one$, so projection by $P$ gives
\eqref{eq:projectedgradient}.

We next determine the kernel.  Fix a vertex $x$ with $d_x\ge3$ and, for
$e\in E(x)$, put
\[
                       p_e^x:=\frac{e^{-y_e}}{D_x}.
\]
Thus $p_e^x>0$ and $\sum_{e\in E(x)}p_e^x=1$.  On the coordinates indexed by
$E(x)$, direct differentiation gives
\[
 \frac{\partial^2\log D_x}{\partial y_e\partial y_f}
 =
 \begin{cases}
  p_e^x(1-p_e^x),&e=f,\\
  -p_e^xp_f^x,&e\ne f.
 \end{cases}
\]
In matrix form this is
\[
             \operatorname{diag}(p^x)-p^x(p^x)^T.
\]
For any $v\in\R^E$, its quadratic form is
\begin{align*}
 \sum_{e\in E(x)}p_e^xv_e^2
 -\left(\sum_{e\in E(x)}p_e^xv_e\right)^2
 &=\frac12\sum_{e,f\in E(x)}p_e^xp_f^x(v_e-v_f)^2\\
 &\ge0.
\end{align*}
This is the variance of the numbers $(v_e)_{e\in E(x)}$ with probabilities
$(p_e^x)_{e\in E(x)}$.  Since every $p_e^x$ is positive, it vanishes exactly when
$v_e=v_f$ for all $e,f\in E(x)$.

The local term in $\Phi$ is multiplied by $d_x-2>0$.  A vertex of degree two
contributes nothing, while at a leaf $x$ with incident edge $e$ one has
$-(\log D_x)=y_e$, which is affine and has zero Hessian.  Consequently,
\[
 v^T\nabla^2\Phi(y)v
 =\frac12\sum_{\substack{x\in V\\d_x\ge3}}(d_x-2)
   \sum_{e,f\in E(x)}p_e^xp_f^x(v_e-v_f)^2.
\]
This formula proves positive semidefiniteness and shows that $v$ belongs to the
kernel precisely when its coordinates are equal on the incident edges of every
vertex of degree at least three.  These are exactly the edges joined in the
conductive line graph.  If that graph is connected, the kernel is
$\operatorname{span}\{\one\}$; hence the restriction of the Hessian to
$\one^\perp$ is positive definite, and $\Phi$ is strictly convex there.
\end{proof}

\begin{remark}\label{rem:nonproper}
Strict convexity does not imply coercivity.  A centered logarithmic trajectory may
leave every compact subset of $\one^\perp$ as some normalized edge weights tend to
zero.  Proposition~\ref{prop:potential} therefore does not by itself imply
convergence of the normalized metric.
\end{remark}

\section{Curvature convergence and its combinatorial description}
\label{sec:dissipation}

\subsection{Maximum principle and dissipation}

 Write $m=|E|,$ and define the average curvature as $\bar\kappa:=1/m\sum_{e\in E}\kappa_e$, which equals to $2/m$ by
Proposition~\ref{prop:GB}.  Define
\begin{equation*}
 V(t):=\sum_{e\in E}(\kappa_e-\bar\kappa)^2,
 \qquad
 \cE(t):=\ip{\kappa}{\cL(w)\kappa}.
\end{equation*}

\begin{proposition}\label{prop:diss}
Along the flow,
\begin{enumerate}[(i)]
\item $\max_e\kappa_e$ is nonincreasing and $\min_e\kappa_e$ is nondecreasing;
\item
\begin{equation}\label{eq:Fdecay}
 \frac{d}{dt}\frac12\sum_e\kappa_e^2=-\cE(t),
 \qquad
 V'(t)=-2\cE(t);
\end{equation}
\item $V$ has a limit and $\int_0^\infty\cE(t)\,dt<\infty$.
\end{enumerate}
\end{proposition}

\begin{proof}
Set
\[
 K_{\max}(t):=\max_{e\in E}\kappa_e(t),
 \qquad
 K_{\min}(t):=\min_{e\in E}\kappa_e(t).
\]
Both functions are locally absolutely continuous, since each is the maximum or
minimum of finitely many $C^1$ functions.  Let $t$ be a differentiability point of
$K_{\max}$, and choose $e$ such that $\kappa_e(t)=K_{\max}(t)$.  The function
$K_{\max}-\kappa_e$ is nonnegative and vanishes at $t$, and is differentiable
there; hence $K_{\max}'(t)=\dot\kappa_e(t)$.  By \eqref{eq:heat},
\[
 K_{\max}'(t)=\sum_{f\ne e}a_{ef}(t)
             \bigl(\kappa_f(t)-\kappa_e(t)\bigr)\le0,
\]
because $a_{ef}\ge0$ and $\kappa_f(t)\le\kappa_e(t)$ for every $f$.  Thus
$K_{\max}'\le0$ almost everywhere, and local absolute continuity shows that
$K_{\max}$ is nonincreasing. By a similar argument, one proves that $K_{\min}$ is nondecreasing. Both extrema
are bounded by Proposition~\ref{prop:treeformula}, so each has a finite limit.


For the first identity in \eqref{eq:Fdecay}, the heat equation gives
\begin{align*}
 \frac{d}{dt}\frac12\sum_{e\in E}\kappa_e^2
 &=\sum_{e\in E}\kappa_e\dot\kappa_e
  =\ip{\kappa}{\dot\kappa} \\
 &=-\ip{\kappa}{\cL\kappa}
  =-\cE(t).
\end{align*}
In particular, \eqref{eq:quadratic} shows that $\cE(t)\ge0$.

Now put $q=\kappa-\bar\kappa\one$.  
Since $\cL\one=0$, we have
\[
 \dot q=\dot\kappa=-\cL\kappa=-\cL q.
\]
Differentiating $V$ gives
\begin{align*}
 V'(t)
 &=2\ip{q}{\dot q}
  =-2\ip{q}{\cL q} \\
 &=-2\ip{\kappa-\bar\kappa\one}{\cL\kappa}
  =-2\ip{\kappa}{\cL\kappa}
  =-2\cE(t),
\end{align*}
where the penultimate equality follows from
$\ip{\one}{\cL\kappa}=\ip{\cL\one}{\kappa}=0$.
$V$ is nonincreasing and bounded below by zero, so it has a finite limit
$V_\infty$.  Integrating the last identity over $[0,T]$ and then letting
$T\to\infty$ yields
\[
 \int_0^\infty\cE(t)\,dt
 =\frac12\bigl(V(0)-V_\infty\bigr)<\infty.
\]
This proves (iii).
\end{proof}

\begin{proposition}\label{prop:Ezero}
For every positive initial metric,
\begin{equation*}
                         \lim_{t\to\infty}\cE(t)=0.
\end{equation*}
Consequently, if a conductive pair $e,f$ satisfies
$\liminf_{t\to\infty}a_{ef}(t)>0$, then
$\kappa_e(t)-\kappa_f(t)\to0$.
\end{proposition}

\begin{proof}
For a conductive pair $e,f$ meeting at $x$, so that $d_x\ge3$, the
arithmetic--geometric mean inequality gives
\[
 D_x\ge \frac1{w_e}+\frac1{w_f}\ge\frac2{\sqrt{w_ew_f}},
 \qquad
 0<a_{ef}\le\frac{d_x-2}{4}.
\]
Since the tree is finite, this bounds $\cL(w)$ uniformly in time.  Direct
differentiation gives
\[
 \frac{\dot a_{ef}}{a_{ef}}
 =\kappa_e+\kappa_f-2\bar\kappa_x,
 \qquad
 \bar\kappa_x:=\frac1{D_x}\sum_{g\ni x}\frac{\kappa_g}{w_g}.
\]
The number $\bar\kappa_x$ is a convex combination of the incident curvatures.
The curvature formula therefore bounds $\kappa$ and $\bar\kappa_x$ uniformly.
Since $a_{ef}$ is also uniformly bounded, the displayed identity gives a uniform
bound for $\dot a_{ef}$.  Differentiating
$\cE=\ip{\kappa}{\cL\kappa}$ and using
\eqref{eq:heat}, we obtain
\[
 \dot\cE
 =-2\norm{\cL\kappa}^2
 +\sum_{\{e,f\}:a_{ef}>0}
   a_{ef}(\kappa_e-\kappa_f)^2
   (\kappa_e+\kappa_f-2\bar\kappa_x).
\]
Hence $\dot\cE$ is bounded, and $\cE$ is uniformly continuous, while
Proposition~\ref{prop:diss} gives $\cE\ge0$ and
$\cE\in L^1(0,\infty)$. By Barbalat's Lemma, $\cE(t)\to0$ as $t\to \infty.$


Finally,
\[
 a_{ef}(t)(\kappa_e(t)-\kappa_f(t))^2\le\cE(t),
\]
which proves the last assertion.
\end{proof}

\subsection{Convergence of edge curvatures and consequences}

\begin{lemma}\label{lem:orderedlimits}
Let $f_1,\ldots,f_N$ be bounded continuous functions on $[0,\infty)$.  For
$1\leq k\leq N$, set
\[
 u_k(t):=\min_{1\leq i_1<i_2<\cdots<i_k\leq N}
 \max\{f_{i_1}(t),f_{i_2}(t),\ldots,f_{i_k}(t)\}.
\]
If $u_k(t)$ has a limit as $t\to\infty$ for every $1\leq k\leq N$, then so
does $f_i(t)$ for every $1\leq i\leq N$.
\end{lemma}

\begin{proof}
Let $\omega(f)$ be the omega-limit set of a function $f$ on $[0,\infty)$ as
$t\to\infty$, that is,
\[
 \omega(f):=\left\{a\in\R:\text{there is a sequence }t_j\to\infty
 \text{ such that }f(t_j)\to a\right\}.
\]
It suffices to prove that $\omega(f_i)$ is a singleton for every
$1\leq i\leq N$.

By assumption,
\[
 \lim_{t\to\infty}u_k(t)=:b_k.
\]
Write $B=\bigcup_{1\leq k\leq N}\{b_k\}$.

We claim that $\omega(f_i)\subset B$ for every $1\leq i\leq N$.  Fix $i$ and
let $a\in\omega(f_i)$.  Choose a sequence $t_j\to\infty$ such that
$f_i(t_j)\to a$.  For each $t$, the numbers $u_1(t),\ldots,u_N(t)$ are the
values $f_1(t),\ldots,f_N(t)$ arranged in nondecreasing order.  Choose a
permutation $\sigma_t\in\mathcal S_N$, the symmetric group on $N$ letters, such
that
\[
 f_l(t)=u_{\sigma_t(l)}(t),\qquad 1\leq l\leq N.
\]
Since $\mathcal S_N$ is finite, after passing to a subsequence, still denoted by
$t_j$, we may assume that $\sigma_{t_j}=\sigma_0$ for every $j$.  Hence
\[
 a=\lim_{j\to\infty}f_i(t_j)
  =\lim_{j\to\infty}u_{\sigma_0(i)}(t_j)
  =b_{\sigma_0(i)}\in B.
\]
This proves the claim.

The omega-limit set of a bounded continuous function
$f:[0,\infty)\to\R$ is connected, since it is the intersection, over $T\geq0$, of
the nested compact intervals $\overline{f([T,\infty))}$.  Therefore
$\omega(f_i)$ is a connected subset of the finite set $B$, and hence is a
singleton.  This proves the result.
\end{proof}

\begin{proof}[\normalfont\bfseries Proof of Theorem~\ref{thm:kappalimit}]
For $1\le i\le m=|E|$, let
\[
 \kappa_{i}(t):=\min_{\substack{I\subseteq E\\ |I|=i}}
   \max_{e\in I}\kappa_e(t),\quad
 S_i(t):
 =\min_{\substack{I\subseteq E\\ |I|=i}}
   \sum_{e\in I}\kappa_e(t).
\]
By definition, $\kappa_i(t)$ is the $i$-th smallest curvature value, and
$S_i(t)=\sum_{j=1}^i\kappa_j(t)$.
For convenience, we set $S_0\equiv0$.

For each $I\subseteq E$ with $|I|=i$, set
\[
 \phi_I(t):=\sum_{e\in I}\kappa_e(t).
\]
Since every $\kappa_e$ is continuously differentiable, $S_i$ is locally Lipschitz and hence locally absolutely
continuous.

Let $t$ be a differentiable point of $S_i$, and choose a minimizing set $I$, so
that $S_i(t)=\phi_I(t)$.  Since $\phi_I-S_i\ge0$ and
$(\phi_I-S_i)(t)=0$, differentiability gives
\[
 S_i'(t)=\phi_I'(t).
\]
Moreover, the minimality of $I$ implies
\[
 \kappa_e(t)\le\kappa_f(t)
 \qquad(e\in I,\ f\notin I),
\]
since otherwise replacing $e$ by $f$ would decrease $\phi_I(t)$.  Using the
symmetry of $a_{ef}$ and the heat equation, we obtain
\begin{align*}
 S_i'(t)
 &=\sum_{e\in I}\dot\kappa_e(t)\\
 &=\sum_{e\in I}\sum_{f\ne e}
     a_{ef}(t)(\kappa_f(t)-\kappa_e(t))\\
 &=\sum_{\substack{e\in I\\f\notin I}}
     a_{ef}(t)(\kappa_f(t)-\kappa_e(t))
 \ge0.
\end{align*}
Indeed, in the second line the terms with $e,f\in I$ cancel in pairs, since
$a_{ef}=a_{fe}$ and
\[
 a_{ef}(\kappa_f-\kappa_e)
 +a_{fe}(\kappa_e-\kappa_f)=0.
\]
Only the pairs with $e\in I$ and $f\notin I$ remain.  Each remaining summand is
nonnegative by the minimality of $I$.
Thus $S_i'\ge0$ almost everywhere.  Since $S_i$ is locally absolutely continuous,
it is nondecreasing.  Proposition~\ref{prop:treeformula} shows that $S_i$ is bounded,
so $S_i$ has a finite limit.  Therefore
\[
 \kappa_{i}(t)=S_i(t)-S_{i-1}(t)
\]
has a finite limit for every $i$.
Lemma~\ref{lem:orderedlimits}, applied to the family $(\kappa_e)_{e\in E}$, now
shows that $c_e=\lim_{t\to\infty}\kappa_e(t)$ exists for every $e\in E$.

For the unnormalized flow,
\[
 \log\widetilde w_e(t)
 =\log\widetilde w_e(0)-\int_0^t\kappa_e(s)\,ds.
\]
Ces\`aro convergence therefore gives
\[
 \lim_{t\to\infty}\frac1t\log\widetilde w_e(t)=-c_e.
\]
Let
\[
 \widetilde S(t):=\sum_{f\in E}\widetilde w_f(t).
\]
The elementary bounds
\[
 \max_f\widetilde w_f(t)\le\widetilde S(t)
 \le m\max_f\widetilde w_f(t),
\]
give
\[
 \max_f\log\widetilde w_f(t)
 \le\log\widetilde S(t)
 \le\max_f\log\widetilde w_f(t)+\log m.
\]
Since the maximum is taken over finitely many edges,
\[
 \lim_{t\to\infty}\frac1t\log\widetilde S(t)
 =\max_f(-c_f)=-c_*.
\]
Using $w_e=\widetilde w_e/\widetilde S$, we obtain
\[
 \lim_{t\to\infty}\frac1t\log w_e(t)=c_*-c_e.
\]


\end{proof}

\begin{remark}\label{rem:minimalclass}
Theorem~\ref{thm:kappalimit} does not rule out oscillation among the edges of
$E_*$.  If $e,f\in E_*$, then
\[
 \frac{d}{dt}\log\frac{w_e}{w_f}=\kappa_f-\kappa_e\longrightarrow0,
\]
but this alone does not imply that $w_e/w_f$ has a limit.  Indeed,
\[
 \log\frac{w_e(t)}{w_f(t)}
 =\log\frac{w_e(0)}{w_f(0)}
  +\int_0^t\bigl(\kappa_f(s)-\kappa_e(s)\bigr)\,ds.
\]
$w_e(t)/w_f(t)$ converges to a finite positive number if and only if the
signed time integral on the right converges to a finite real number.  This does
not require absolute integrability of $\kappa_f-\kappa_e$.  The dissipation
identity gives only a weighted $L^2$ estimate and, when the conductances
degenerate, does not imply convergence of this time integral.
\end{remark}

\subsection{Combinatorial structure for curvature limits}

We next characterize the limiting curvature vector in terms of the
combinatorics of the tree.
\begin{proof}[\normalfont\bfseries Proof of Theorem~\ref{thm:combinatorial}]
For a fixed vertex $x$, put
\[
                 g_x(A):=\one_{\{E(x)\subseteq A\}}.
\]
We verify that $g_x$ is supermodular.  For any $A,B\subseteq E$,
\begin{equation}\label{eq:gxsupermodular}
 g_x(A)+g_x(B)\le g_x(A\cup B)+g_x(A\cap B).
\end{equation}
Indeed, if $E(x)$ is contained in both $A$ and $B$, both sides are equal to
$2$.  If it is contained in exactly one of them, both sides are equal to $1$.
If it is contained in neither, the left-hand side is zero and the inequality
is immediate.  When $d_x\ge3$, the coefficient $2-d_x$ is negative, so
\eqref{eq:gxsupermodular} shows that $(2-d_x)g_x$ is submodular.  If $x$ is a
leaf, then $E(x)$ consists of one edge and $g_x$ is a modular single-edge
function.  A vertex of degree two contributes zero.  Since $\beta$ is the sum
of these vertex terms, it is submodular.

Let $A\subseteq E$ be nonempty.  Summing the curvature formula over $A$ and
grouping the terms by vertices gives
\begin{equation}\label{eq:subsetcurvature}
 \sum_{e\in A}\kappa_e(t)
 =\sum_{\substack{x\in V\\A\cap E(x)\ne\varnothing}}
 (2-d_x)
 \frac{\displaystyle\sum_{e\in A\cap E(x)}w_e(t)^{-1}}
      {\displaystyle\sum_{f\in E(x)}w_f(t)^{-1}}.
\end{equation}
If $E(x)\subseteq A$, the contribution of $x$ is $2-d_x$.  If $A$ contains
only some of the edges at $x$, then either $d_x=2$ and the contribution is zero,
or $d_x\ge3$ and the contribution is negative.  A leaf cannot occur in the
second case.  Therefore
\begin{equation*}
                    \sum_{e\in A}\kappa_e(t)\le\beta(A).
\end{equation*}
Letting $t\to\infty$ yields $c(A)\le\beta(A)$.  Proposition~\ref{prop:GB}
gives $c(E)=2=\beta(E)$, and therefore $c\in B(\beta)$.

To prove \eqref{eq:tightlevels},
fix $j,$ and we have  
\eqref{eq:subsetcurvature} with $A=A_j.$
Consider a vertex $x$ for
which $A_j\cap E(x)$ is nonempty and is not all of $E(x)$.  The only case that
needs consideration is $d_x\ge3$.  Choose $g\in E(x)\setminus A_j$.  For every
$e\in A_j\cap E(x)$, we have $c_e\le\alpha_j<c_g$, and
\eqref{eq:exponentialrates} gives
\[
 \frac{w_g(t)}{w_e(t)}\longrightarrow0.
\]
Since $D_x(t)\ge w_g(t)^{-1}$,
\[
 \frac{\displaystyle\sum_{e\in A_j\cap E(x)}w_e(t)^{-1}}
      {\displaystyle\sum_{f\in E(x)}w_f(t)^{-1}}
 \le\sum_{e\in A_j\cap E(x)}\frac{w_g(t)}{w_e(t)}
 \longrightarrow0.
\]
Every partial contribution in \eqref{eq:subsetcurvature} vanishes in the
limit.  The full contributions sum to $\beta(A_j)$, proving
\eqref{eq:tightlevels}.

The vector $c$ has the decomposition
\[
 c=\alpha_r\one+
   \sum_{j=1}^{r-1}(\alpha_j-\alpha_{j+1})\one_{A_j}.
\]
Let $z\in B(\beta)$.  Since $z(E)=c(E)$ and, by
\eqref{eq:tightlevels},
\[
                         z(A_j)\le\beta(A_j)=c(A_j),
\]
we obtain
\begin{align*}
 \langle c,z-c\rangle
 &=\alpha_r\bigl(z(E)-c(E)\bigr)\\
 &\quad+
   \sum_{j=1}^{r-1}(\alpha_j-\alpha_{j+1})
                    \bigl(z(A_j)-c(A_j)\bigr)\ge0.
\end{align*}
Consequently,
\[
 \frac12\norm{z}^2-\frac12\norm{c}^2
 =\langle c,z-c\rangle+\frac12\norm{z-c}^2\ge0.
\]
Thus $c$ minimizes the objective in \eqref{eq:minimumnormbase}.  The last
inequality is strict when $z\ne c$, so the minimizer is unique.

For the recursive description, let $A_j\subsetneq A\subseteq E$.  Then
every edge in $A\setminus A_j$ has limiting curvature at least
$\alpha_{j+1}$.  Using the tightness of $A_j$ and the inequality
$c(A)\le\beta(A)$, we obtain
\[
 \alpha_{j+1}(|A|-|A_j|)
 \le c(A)-c(A_j)
 \le\beta(A)-\beta(A_j).
\]
Equality holds when $A=A_{j+1}$, so \eqref{eq:principalchain} follows.  If
another set $A$ attains the same minimum, then equality holds in the first
inequality, and every edge of $A\setminus A_j$ has curvature
$\alpha_{j+1}$.  Hence $A\subseteq A_{j+1}$, which proves the maximality
assertion.

Taking $j=0$ in \eqref{eq:principalchain} gives
\eqref{eq:cstarcombinatorial}, and $A_1=E_*$.  If a nonempty set $A$ attains
this minimum, then
\[
 |A|c_*\le c(A)\le\beta(A)=|A|c_*.
\]
Since every $c_e\ge c_*$, all edges of $A$ belong to $E_*$.  Hence $E_*$ is the
unique largest minimizing set, and \eqref{eq:Estarunion} follows.  Finally, put
$F(A)=\beta(A)-c_*|A|$.  The function $F$ is submodular and nonnegative.  If
$F(A)=F(B)=0$, submodularity gives
\[
                  F(A\cup B)+F(A\cap B)\le F(A)+F(B)=0.
\]
Both terms on the left are nonnegative, so they vanish.  This proves the last
assertion, with the stated qualification when $A\cap B$ is empty.
\end{proof}

\section{Limiting normalized weights}\label{sec:weights}
\subsection{Equations for omega-limit points}\label{subsec:omega-limit-equations}
\begin{proof}[\normalfont\bfseries Proof of Theorem~\ref{thm:algebraiccriterion}]
Choose a sequence $t_j\to\infty$ such that $w(t_j)\to w^*$, and put
$C=E^+(w^*)$ and $p=w^*|_C$.  Then $p\in\Delta_C$.

Let $e\in C$ and let $x$ be one of its endpoints.  If $E(x)\subseteq C$, all
weights incident to $x$ have positive limits along this sequence.  Hence
\[
 \frac{2-d_x}{w_e(t_j)D_x(t_j)}
 \longrightarrow
 \frac{2-d_x}{p_e\displaystyle\sum_{f\in E(x)}p_f^{-1}}.
\]
If $E(x)\not\subseteq C$, choose $g\in E(x)\setminus C$.  Since
$w_g(t_j)\to0$ whereas $w_e(t_j)\to p_e>0$,
\[
 w_e(t_j)D_x(t_j)
 \ge \frac{w_e(t_j)}{w_g(t_j)}\longrightarrow\infty.
\]
In this case the contribution of $x$ tends to zero.  Applying the same argument at
the other endpoint of $e$ gives
\[
                         \kappa_e(t_j)\longrightarrow\kappa_e^C(p).
\]
By Theorem~\ref{thm:kappalimit}, $C\subseteq E_*$ and
$\kappa_e(t_j)\to c_*$ for every $e\in C$.  Therefore
$p\in\mathcal S_C(c_*)$.

For the convergence assertion, let $T\ge0$ and set
\[
 K_T:=\overline{\{w(t):t\ge T\}}\subseteq\overline\Delta_E.
\]
The set $\{w(t):t\ge T\}$ is connected because $w$ is continuous, and therefore
$K_T$ is connected.  The sets $K_T$ are compact and decrease as $T$ increases, and
\[
                         \omega(w)=\bigcap_{T\ge0}K_T.
\]
It follows that $\omega(w)$ is nonempty and connected.  We have already shown that
it is contained in \eqref{eq:candidates}.  If that set is finite, then
$\omega(w)$ must consist of a single point.  This is equivalent to the convergence
of $w(t)$.
\end{proof}

We first describe the solutions on a fixed support.

\begin{proof}[\normalfont\bfseries Proof of Theorem~\ref{thm:boundarystructure}]
Let $p\in\mathcal S_C(c)$.  On summing \eqref{eq:boundarycurvature} over
$e\in C$, the contribution of a vertex $x$ with $E(x)\subseteq C$ is
\[
 \sum_{e\in E(x)}
 \frac{2-d_x}{p_e\displaystyle\sum_{f\in E(x)}p_f^{-1}}
 =2-d_x.
\]
Summing over all such vertices gives \eqref{eq:balanceidentity}.

It remains to describe the solutions with a fixed support.  Write $y_e=\log p_e$
and define
\[
 \Phi_C(y):=
 \sum_{\substack{x\in V\\E(x)\subseteq C}}
 (d_x-2)\log\left(\sum_{e\in E(x)}e^{-y_e}\right).
\]
The term associated with a leaf is affine, and a vertex of degree two contributes
zero.  The other terms are convex.  As in Proposition~\ref{prop:potential}, direct
differentiation gives
\[
                         \nabla\Phi_C(y)=\kappa^C(p).
\]
Let $v\in\R^C$.  At a vertex $x$ of degree at least three, the contribution of the
corresponding term to the Hessian quadratic form is
\[
 (d_x-2)\sum_{\substack{\{e,f\}\subseteq E(x)}}
 \frac{e^{-y_e-y_f}}{D_x(y)^2}(v_e-v_f)^2.
\]
All coefficients in this sum are positive.  Hence this contribution vanishes if
and only if $v_e$ has the same value for all $e\in E(x)$.  Summing over the
vertices that contribute to $\Phi_C$, we obtain
\begin{equation}\label{eq:kernelGC}
 \ker\nabla^2\Phi_C(y)
 =\{v\in\R^C:v\text{ is constant on each }C_i\}.
\end{equation}
Notice that this kernel does not depend on $y$.

Fix one solution $p^0\in\mathcal S_C(c)$.  For each $i$, let
\begin{equation}\label{eq:blockprofilelimit}
 q_e^{(i)}:=\frac{p_e^0}{\sum_{f\in C_i}p_f^0},
 \qquad e\in C_i.
\end{equation}
Every vertex that makes a nonzero contribution to
\eqref{eq:boundarycurvature} has all its incident edges in a single component
$C_i$; at a leaf there is only one incident edge.
Multiplying the weights in each $C_i$ by an arbitrary positive
constant does not change any $\kappa_e^C$.  It follows that every vector on the
right-hand side of \eqref{eq:SCsimplex} belongs to $\mathcal S_C(c)$.

Conversely, let $p\in\mathcal S_C(c)$ and put $y=\log p$ and $y^0=\log p^0$.
Every component of both gradients equals $c$, so their difference is zero.  The
fundamental theorem of calculus gives
\begin{align*}
 0
 &=\left\langle y-y^0,
       \nabla\Phi_C(y)-\nabla\Phi_C(y^0)\right\rangle\\
 &=\int_0^1
   \left\langle y-y^0,
   \nabla^2\Phi_C\bigl(y^0+\tau(y-y^0)\bigr)(y-y^0)
   \right\rangle\,d\tau.
\end{align*}
The integrand is continuous and nonnegative.  It therefore vanishes for every
$\tau\in[0,1]$, and \eqref{eq:kernelGC} shows that $y-y^0$ is constant on each
$C_i$.  Hence $p$ is obtained from $p^0$ by rescaling its restrictions to the
components and then imposing $\sum_{e\in C}p_e=1$.  If $\widehat p^0$ is another
solution, the same argument gives $\widehat p_e^0=r_i p_e^0$ on $C_i$ for some
$r_i>0$.  After normalization on $C_i$,
\[
 \frac{\widehat p_e^0}{\sum_{f\in C_i}\widehat p_f^0}
 =\frac{p_e^0}{\sum_{f\in C_i}p_f^0},
 \qquad e\in C_i.
\]
This proves \eqref{eq:SCsimplex} and shows that each profile
\eqref{eq:blockprofilelimit} is uniquely determined.

Finally, let $w^*\in\omega(w)$ have support $C$ of maximal cardinality, and assume
that $G_C$ is connected.  By Theorem~\ref{thm:algebraiccriterion},
$w^*|_C\in\mathcal S_C(c_*)$, and the part just proved shows that this set consists
of one point.  All coordinates of $w^*$ on $C$ are positive.  Hence, in a sufficiently
small neighborhood of $w^*$, every point of $\omega(w)$ has support containing
$C$.  Maximality of $|C|$ forces that support to equal $C$.  Every such omega-limit
point belongs to $\mathcal S_C(c_*)$, which is a singleton, and hence equals
$w^*$.  Therefore $w^*$ is isolated in
$\omega(w)$.  The proof of Theorem~\ref{thm:algebraiccriterion} showed that
$\omega(w)$ is connected; it follows that $\omega(w)=\{w^*\}$, and hence
$w(t)\to w^*$.  If $C=E_*$, maximality follows from
\eqref{eq:supportminimal}, which proves the last assertion.
\end{proof}

\subsection{Existence of positive boundary solutions}
\begin{proof}[\normalfont\bfseries Proof of Theorem~\ref{thm:positivecriterion}]
Suppose first that
$p\in\mathcal S_C(c)$, and let $A$ be a nonempty proper subset of $C$.  Since
$G_C$ is connected, some vertex $x$ of degree at least three has
$E(x)\subseteq C$ and has incident edges in both $A$ and $C\setminus A$.  When
$\sum_{e\in A}\kappa_e^C(p)$ is grouped by vertices, this $x$ makes a strictly
negative partial contribution.  All other partial contributions are
nonpositive, while the vertices whose incident edges are all in $A$ contribute
$\beta(A)$.  Consequently,
\[
                         |A|c<\beta(A),
\]
which is equivalent to $\lambda(A)>c$.

Conversely, assume that $\lambda(A)>c$ for every nonempty proper
$A\subsetneq C$, and define
\[
 F_C(u)=\Phi_C(u)-c\sum_{e\in C}u_e.
\]
The identity $c=\lambda(C)$ shows that $F_C$ is unchanged when a constant is
added to all coordinates.  We claim that it is coercive on
\[
 H_0^C:=\left\{u\in\R^C:\sum_{e\in C}u_e=0\right\}.
\]
If $|C|=1$, this space consists of one point and the conclusion is immediate.
Assume henceforth that $|C|\ge2$.
Fix $u\in H_0^C$ and put
\[
 m_u:=\min_{e\in C}u_e,
 \qquad \widehat u:=u-m_u\one,
 \qquad M:=\max_{e\in C}\widehat u_e.
\]
For $0\le t<M$, set
\[
                         A_t:=\{e\in C:\widehat u_e>t\}.
\]
Every $A_t$ is a nonempty proper subset of $C$.  Since there are only finitely
many such subsets,
\[
 \delta:=\min_{\varnothing\ne A\subsetneq C}
          \bigl(\beta(A)-c|A|\bigr)>0.
\]
All coordinates of $\widehat u$ are nonnegative and at least one is zero.  For
every vertex appearing in the definition of $\Phi_C$,
\[
 \min_{e\in E(x)}\widehat u_e
 =\int_0^M\one_{\{E(x)\subseteq A_t\}}\,dt,
\]
because $E(x)\subseteq A_t$ precisely when
$t<\min_{e\in E(x)}\widehat u_e$, apart from endpoints of intervals, which do
not affect the integral.  Similarly,
\[
                \sum_{e\in C}\widehat u_e=\int_0^M|A_t|\,dt.
\]
Define
\[
 R(v):=\sum_{\substack{x\in V\\E(x)\subseteq C}}
       (2-d_x)\min_{e\in E(x)}v_e-c\sum_{e\in C}v_e.
\]
Since $\beta(C)=c|C|$, both $F_C$ and $R$ are unchanged when a constant is
added to every coordinate.  Summing the preceding identities therefore gives
\begin{align*}
 R(u)=R(\widehat u)
 &=\int_0^M\bigl(\beta(A_t)-c|A_t|\bigr)\,dt
 \ge\delta M.
\end{align*}
For $d_x\ge3$,
\[
 (d_x-2)\log\left(\sum_{e\in E(x)}e^{-v_e}\right)
 \ge (2-d_x)\min_{e\in E(x)}v_e.
\]
At a leaf the two sides are equal, and a vertex of degree two contributes zero.
It follows that $F_C(v)\ge R(v)$ for every $v\in\R^C$.  Since $u\in H_0^C$,
$\min_eu_e\le0\le\max_eu_e$, and hence
\[
 \max_e|u_e|\le \max_eu_e-\min_eu_e.
\]
Thus a bounded coordinate range would imply that $u$ is bounded.  After
subtracting the minimum coordinate, $M$ is exactly this range, so $M\to\infty$
whenever $\norm{u}\to\infty$ in $H_0^C$.  Together with
$F_C(u)\ge R(u)\ge\delta M$, this proves coercivity.

The restriction of $F_C$ to $H_0^C$ therefore has a minimizer.  It is strictly
convex there by \eqref{eq:kernelGC} and the connectedness of $G_C$, so the
minimizer is unique.  At the minimizer, the directional derivative vanishes in
every direction in $H_0^C$.  Hence the gradient of $F_C$ belongs to
$(H_0^C)^\perp=\operatorname{span}\{\one\}$.  On the other hand, translation
invariance gives $\langle\nabla F_C,\one\rangle=0$.  The gradient is therefore
zero.  Exponentiating
and normalizing the minimizer produces $p\in\Delta_C$ with
$\kappa_e^C(p)=c$ for every $e\in C$.  Thus
$\mathcal S_C(c)\ne\varnothing$; its uniqueness also follows from
Theorem~\ref{thm:boundarystructure}.

For the last assertion, Theorem~\ref{thm:algebraiccriterion} gives
$p:=w^*|_C\in\mathcal S_C(c_*)$, and the balance identity gives
$\lambda(C)=c_*$.  The boundary equations decouple across the components of
$G_C$.  Indeed, if a vertex of degree at least three contributes to
\eqref{eq:boundarycurvature}, then all its incident edges lie in $C$ and are
pairwise adjacent in $G_C$, so they belong to one component.  A leaf has only
one incident edge, and a vertex of degree two contributes zero.  Thus no
nonzero term in the boundary equations involves edges from two different
components.  After normalizing $p|_{C_i}$, scale invariance therefore gives a point of
$\mathcal S_{C_i}(c_*)$, so $\lambda(C_i)=c_*$.  The necessary part just proved
shows that no nonempty proper subset of $C_i$ has the same value of $\lambda$.
The conclusion follows from \eqref{eq:cstarcombinatorial}.
\end{proof}

\begin{proof}[\normalfont\bfseries Proof of Corollary~\ref{cor:Estarcriterion}]
A vertex that makes a
nonzero contribution to $\beta(E_*)$ belongs to exactly one component $C_i$:
this is immediate for a leaf, while the incident edges at a contributing vertex
of degree at least three form a clique in $G_{E_*}$.  Vertices of degree two
contribute zero.  Therefore
\[
                         \beta(E_*)=\sum_{i=1}^s\beta(C_i).
\]
By Theorem~\ref{thm:combinatorial},
$\beta(E_*)=c_*|E_*|$ and $\beta(C_i)\ge c_*|C_i|$.  Equality of the sums forces
\eqref{eq:componentbalance} for every $i$.

The boundary curvature equations on distinct components do not interact.  More
precisely, scale invariance gives
\[
 \mathcal S_{E_*}(c_*)\ne\varnothing
 \quad\Longleftrightarrow\quad
 \mathcal S_{C_i}(c_*)\ne\varnothing
 \quad\text{for every }i.
\]
Applying \eqref{eq:strictcriterion} to each connected component proves
\eqref{eq:componentcriterion}.
\end{proof}

\subsection{Convergence with full minimal support}
\begin{proof}[\normalfont\bfseries Proof of Theorem~\ref{thm:fullsupport}]
Set $C=E_*$.  We use \eqref{eq:boundarycurvature} for arbitrary positive vectors
on $C$, without requiring their coordinates to sum to one; $\kappa^C$ is
invariant under a common rescaling.

Suppose first that $C\ne E$.  Since $E$ is finite,
\[
 \eta:=\min_{g\notin C}(c_g-c_*)>0.
\]
Choose $0<\delta<\eta$.  Theorem~\ref{thm:kappalimit} gives, for every $e\in C$
and $g\notin C$,
\[
 \frac1t\log\frac{\widetilde w_g(t)}{\widetilde w_e(t)}
 \longrightarrow-(c_g-c_*),
\]
and hence, uniformly over the finitely many such pairs,
\begin{equation}\label{eq:outsideexp}
             \frac{\widetilde w_g(t)}{\widetilde w_e(t)}
             =O(e^{-\delta t}).
\end{equation}
For $e\in C$, compare its actual curvature with $\kappa_e^C(\widetilde w_C)$.
The two expressions agree at an endpoint $x$ for which $E(x)\subseteq C$.  If
$E(x)\not\subseteq C$, choose $g\in E(x)\setminus C$.  The contribution of $x$
to the actual curvature satisfies
\[
 \left|\frac{2-d_x}{\widetilde w_eD_x}\right|
 \le |2-d_x|\frac{\widetilde w_g}{\widetilde w_e}
 =O(e^{-\delta t}).
\]
Summing the two endpoint estimates gives
\begin{equation}\label{eq:curvatureerror}
 \kappa_e(t)=\kappa_e^C(\widetilde w_C(t))+\rho_e(t),
 \qquad \norm{\rho(t)}=O(e^{-\delta t}).
\end{equation}
If $C=E$, the two curvature expressions agree identically, and we set
$\rho\equiv0$.  In either case, $\rho(t)\to0$ and
$\int_0^\infty\norm{\rho(t)}\,dt<\infty$.

Remove the common exponential rate by putting
\[
 u_e(t):=\log\widetilde w_e(t)+c_*t,
 \qquad e\in C,
\]
and define the convex function
\[
 F_C(u):=\Phi_C(u)-c_*\sum_{e\in C}u_e.
\]
Since $\nabla\Phi_C=\kappa^C$, equations \eqref{eq:unnormalized} and
\eqref{eq:curvatureerror} give
\begin{equation}\label{eq:perturbedgradient}
             u'=-\nabla F_C(u)+\varepsilon(t),
             \qquad \varepsilon(t):=-\rho(t).
\end{equation}
By assumption, choose $q\in\mathcal S_C(c_*)$ and put $v=\log q$.  Then
$\nabla F_C(v)=0$.

We first show that $u$ is bounded.  For $a>0$, set
\[
 R_a(t):=\bigl(\norm{u(t)-v}^2+a^2\bigr)^{1/2}.
\]
Convexity of $F_C$ and \eqref{eq:perturbedgradient} imply
\begin{align*}
 R_a'(t)
 &=\frac{-\langle u-v,\nabla F_C(u)-\nabla F_C(v)\rangle
          +\langle u-v,\varepsilon\rangle}{R_a(t)}\\
 &\le \norm{\varepsilon(t)}.
\end{align*}
Integrating from $s$ to $t$ and letting $a\downarrow0$ yields
\begin{equation}\label{eq:fejer}
 \norm{u(t)-v}
 \le \norm{u(s)-v}+\int_s^t\norm{\varepsilon(\tau)}\,d\tau,
 \qquad t\ge s.
\end{equation}
This bounds $u(t)$.  By Theorem~\ref{thm:kappalimit},
\[
 \nabla F_C(u(t))
 =\kappa(t)|_C-c_*\one-\rho(t)\longrightarrow0
\]
as $t\to\infty$.  Since $u(t)$ is bounded and $C$ is finite, there exits a sequence $t_j\to\infty$ and a vector $u^\infty$ such that
\[
u(t_j)\to u^\infty.
\]
The gradient is continuous, and hence
\[
 \nabla F_C(u^\infty)
 =\lim_{j\to\infty}\nabla F_C(u(t_j))=0.
\]
Thus $u^\infty$ is a critical point of $F_C$.  Applying \eqref{eq:fejer} with
$v=u^\infty$ gives, for $t\ge t_j$,
\[
 \norm{u(t)-u^\infty}
 \le \norm{u(t_j)-u^\infty}
    +\int_{t_j}^\infty\norm{\varepsilon(\tau)}\,d\tau.
\]
The right-hand side tends to zero as $j\to\infty$.  Hence
$u(t)\to u^\infty$.

Set $A_e=e^{u_e^\infty}>0$ for $e\in C$.  Then
\[
 \widetilde w_e(t)=e^{-c_*t}(A_e+o(1)),
 \qquad e\in C.
\]
If $g\notin C$, fix an edge $e\in C$.  Equations
\eqref{eq:outsideexp} and the preceding asymptotic formula give
\[
 e^{c_*t}\widetilde w_g(t)
 =\frac{\widetilde w_g(t)}{\widetilde w_e(t)}
   \bigl(e^{c_*t}\widetilde w_e(t)\bigr)
 \longrightarrow0.
\]
After normalization, we obtain
\[
 w_e(t)\longrightarrow w_e^\infty:=
 \begin{cases}
 \displaystyle\frac{A_e}{\sum_{f\in C}A_f},&e\in C,\\[6pt]
 0,&e\notin C.
 \end{cases}
\]
Since $\nabla F_C(u^\infty)=0$, scale invariance gives
$w^\infty|_C\in\mathcal S_C(c_*)$.  This proves all assertions except the last.
If an omega-limit point $p$ has support $E_*$, Theorem~\ref{thm:algebraiccriterion}
gives $p|_{E_*}\in\mathcal S_{E_*}(c_*)$, so the result just proved applies; its
limit must equal the given  omega-limit point $p$.
\end{proof}

\begin{example}\label{ex:nestedminimizer}
Let $x$ and $y$ be adjacent vertices with $d_x=3$ and $d_y=4$, and write
$z=xy$.  Join the other two edges at $x$ to degree-two vertices $a_1,a_2$, and
attach a leaf to each $a_i$.  At $y$, join two edges to degree-two vertices
$b_1,b_2$, again followed by leaves, and join the remaining edge directly to a
leaf $q$.  This tree is shown in Figure~\ref{fig:nestedminimizer}.  It is not a
caterpillar, since $x$ has three neighbors in the subgraph induced by the
non-leaf vertices.

Set
\[
 C:=\{z,xa_1,xa_2,yb_1,yb_2,yq\},
 \qquad
 A:=\{z,xa_1,xa_2\}.
\]
Then
\[
 |C|=6,\qquad \beta(C)=(2-d_x)+(2-d_y)+1=-2,
 \qquad \lambda(C)=-\frac13,
\]
where the last $1$ is the contribution of the leaf $q$.  On the other hand,
\[
 |A|=3,\qquad \beta(A)=2-d_x=-1,
 \qquad \lambda(A)=-\frac13.
\]

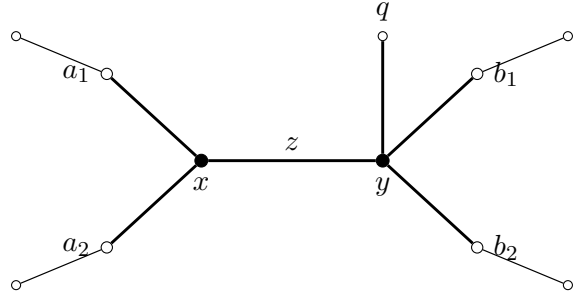
\begin{figure}[ht]
\centering
\begin{tikzpicture}[
  branch/.style={circle,draw,fill=black,inner sep=1.7pt},
  vertex/.style={circle,draw,fill=white,inner sep=1.5pt},
  leaf/.style={circle,draw,fill=white,inner sep=1.2pt},
  core/.style={line width=1.1pt},
  outer/.style={line width=0.45pt}]
\node[branch,label=below:$x$] (x) at (-1.2,0) {};
\node[branch,label=below:$y$] (y) at (1.2,0) {};
\node[vertex,label=left:$a_1$] (a1) at (-2.45,1.15) {};
\node[vertex,label=left:$a_2$] (a2) at (-2.45,-1.15) {};
\node[leaf] (al1) at (-3.65,1.65) {};
\node[leaf] (al2) at (-3.65,-1.65) {};
\node[vertex,label=right:$b_1$] (b1) at (2.45,1.15) {};
\node[vertex,label=right:$b_2$] (b2) at (2.45,-1.15) {};
\node[leaf] (bl1) at (3.65,1.65) {};
\node[leaf] (bl2) at (3.65,-1.65) {};
\node[leaf,label=above:$q$] (q) at (1.2,1.65) {};
\draw[core] (x)--node[above] {$z$} (y);
\draw[core] (x)--(a1);
\draw[core] (x)--(a2);
\draw[outer] (a1)--(al1);
\draw[outer] (a2)--(al2);
\draw[core] (y)--(b1);
\draw[core] (y)--(b2);
\draw[outer] (b1)--(bl1);
\draw[outer] (b2)--(bl2);
\draw[core] (y)--(q);
\end{tikzpicture}
\caption{A non-caterpillar with a proper nested minimizing edge set.  The thick
edges form $C$, and the three thick edges incident to $x$ form $A$.}
\label{fig:nestedminimizer}
\end{figure}

We check that no nonempty edge set $B$ has $\lambda(B)<-1/3$.  The only negative
terms in $\beta(B)$ can come from $x$ and $y$.  If neither $E(x)$ nor $E(y)$ is
contained in $B$, then $\beta(B)\ge0$.  If only $E(x)$ is contained in $B$, then
$\beta(B)\ge-1$ and $|B|\ge3$.  If only $E(y)$ is contained in $B$, the edge
$yq$ also includes the leaf contribution at $q$, so $\beta(B)\ge-1$ and
$|B|\ge4$.  If both are contained in $B$, then $\beta(B)\ge-2$ and $|B|\ge6$.
In every case $\lambda(B)\ge-1/3$.  Theorem~\ref{thm:combinatorial} therefore
gives
\[
                              c_*=-\frac13.
\]
Both $A$ and $C$ attain the minimum, so $C\subseteq E_*$.  Every edge outside
$C$ joins a degree-two vertex to a leaf and consequently has curvature identically
equal to $1$.  Hence $E_*=C$.

The graph $G_C$ is connected, but the proper subset $A\subsetneq C$ satisfies
$\lambda(A)=\lambda(C)=c_*$.  The strict criterion
\eqref{eq:strictcriterion} now gives
\[
                         \mathcal S_{E_*}(c_*)=\varnothing.
\]
In particular, $c_*<0$ does not imply the hypothesis of Theorem~\ref{thm:fullsupport},
even for a non-caterpillar tree.
\end{example}

\subsection{Comparison of trajectories}

\begin{proof}[\normalfont\bfseries Proof of Theorem~\ref{thm:trajectorycomparison}]
Put
\[
 y=\log\widetilde w,\qquad z=\log\widetilde v,
 \qquad h=z-y.
\]
By Proposition~\ref{prop:potential},
\[
 h'=-\bigl(\nabla\Phi(z)-\nabla\Phi(y)\bigr)=-B(t)h,
 \qquad
 B(t):=\int_0^1\nabla^2\Phi\bigl(y+\tau h\bigr)\,d\tau.
\]
Each Hessian in the integral is a symmetric weighted Laplacian.  Hence $B(t)$ is
also a symmetric weighted Laplacian, say
\[
 (B(t)a)_e=\sum_{f\ne e}b_{ef}(t)(a_e-a_f),
 \qquad b_{ef}(t)=b_{fe}(t)\ge0.
\]
Set $H_{\max}(t)=\max_e h_e(t)$.  At a differentiability point of
$H_{\max}$, choose an edge $e$ with $h_e=H_{\max}$.  Then
\[
 H_{\max}'(t)=h_e'(t)
 =\sum_{f\ne e}b_{ef}(t)(h_f(t)-h_e(t))\le0.
\]
Since $H_{\max}$ is locally absolutely continuous, it is nonincreasing.  The
same argument at an edge attaining the minimum shows that
$H_{\min}(t):=\min_e h_e(t)$ is nondecreasing.  In particular, every $h_e$ is
bounded between $H_{\min}(0)$ and $H_{\max}(0)$.

For $1\le i\le m=|E|$, set
\[
 u_i(t):=
 \min_{\substack{I\subseteq E\\ |I|=i}}
 \max_{e\in I}h_e(t),
 \qquad
 Q_i(t):=
 \min_{\substack{I\subseteq E\\ |I|=i}}
 \sum_{e\in I}h_e(t),
\]
and put $Q_0\equiv0$.  For each fixed $t$, $u_i(t)$ is the
$i$-th smallest value among $\{h_e(t):e\in E\}$.  Therefore
\[
 Q_i(t)=\sum_{j=1}^i u_j(t),
 \qquad
 u_i(t)=Q_i(t)-Q_{i-1}(t).
\]


Each $Q_i$ is locally Lipschitz.  At a differentiable
point $t$ of $Q_i$, choose a minimizing set $I$.  If
$\psi_I=\sum_{e\in I}h_e$, then $\psi_I-Q_i\ge0$ and this difference vanishes
at the chosen time.  Hence $Q_i'=\psi_I'$ there.  The minimality of $I$ also
implies
\[
                         h_e(t)\le h_f(t)
                         \qquad(e\in I,\ f\notin I),
\]
since otherwise exchanging $e$ and $f$ would decrease $\psi_I(t)$.  Using
$h_e'=\sum_{f\ne e}b_{ef}(h_f-h_e)$, we have
\[
 Q_i'(t)
 =\sum_{\substack{e\in I\\f\notin I}}
   b_{ef}(t)\bigl(h_f(t)-h_e(t)\bigr)\ge0.
\]
Thus $Q_i'\ge0$ almost everywhere.  Since $Q_i$ is locally
absolutely continuous, it is nondecreasing.  Moreover,
\[
 iH_{\min}(0)\le Q_i(t)\le iH_{\max}(0),
\]
so $Q_i$ has a finite limit.  Since
\[
 u_i(t)=Q_i(t)-Q_{i-1}(t),
\]
each $u_i$ also has a finite limit.  Lemma~\ref{lem:orderedlimits}
then implies that $h_e(t)$ has a finite limit for every $e\in E$.
Exponentiating gives
\[
 \frac{\widetilde v_e(t)}{\widetilde w_e(t)}=e^{h_e(t)}
 \longrightarrow r_e\in(0,\infty).
\]

Let $w$ and $v$ be the normalizations of $\widetilde w$ and $\widetilde v$,
respectively, and put $a(t)=(e^{h_e(t)})_{e\in E}$.  Then $a(t)\to r$ and
\[
                         v(t)=T_{a(t)}(w(t)).
\]
Since $r_e>0$ for every $e$, the vectors $a(t)$ eventually lie in a compact
subset of $(0,\infty)^E$.  The map
\[
 (a,p)\longmapsto T_a(p)
\]
is uniformly continuous on the product of this compact set with
$\overline\Delta_E$.  Hence
\[
                         \norm{v(t)-T_r(w(t))}\longrightarrow0.
\]

If $p\in\omega(w)$, choose $t_j\to\infty$ with $w(t_j)\to p$.  Then
$v(t_j)\to T_r(p)$, and therefore
$T_r(\omega(w))\subseteq\omega(v)$.  Conversely, if $q\in\omega(v)$, choose
$t_j\to\infty$ with $v(t_j)\to q$.  Compactness of $\overline\Delta_E$ gives,
after passing to a subsequence, $w(t_j)\to p$ for some $p\in\omega(w)$.  The same
limit now gives $q=T_r(p)$.  This proves \eqref{eq:omegaprojective}.

The inverse of $T_r$ is $T_{r^{-1}}$, so $T_r$ is a homeomorphism.  Since all
$r_e$ are positive, $E^+(T_r(p))=E^+(p)$.  The remaining assertions follow.
\end{proof}

\subsection{Examples and the remaining problem}
\begin{example}\label{ex:cuttingedge}
Fix $n\ge3$.  For each $\sigma\in\{-,+\}$, take a vertex $h_\sigma$ of degree
$n$, join it to vertices $q_\sigma,v_{\sigma,1},\ldots,v_{\sigma,n-1}$, and
attach a leaf $p_{\sigma,j}$ to each $v_{\sigma,j}$.  Finally, join $q_-$ to
$q_+$ by an edge $z$.  Every $q_\sigma$ and $v_{\sigma,j}$ has degree two.
Set
\[
 C_\sigma:=\{h_\sigma q_\sigma\}
   \cup\{h_\sigma v_{\sigma,j}:1\le j\le n-1\},
 \qquad
 L_\sigma:=\{v_{\sigma,j}p_{\sigma,j}:1\le j\le n-1\}.
\]
The tree is shown in Figure~\ref{fig:cuttingedge}.

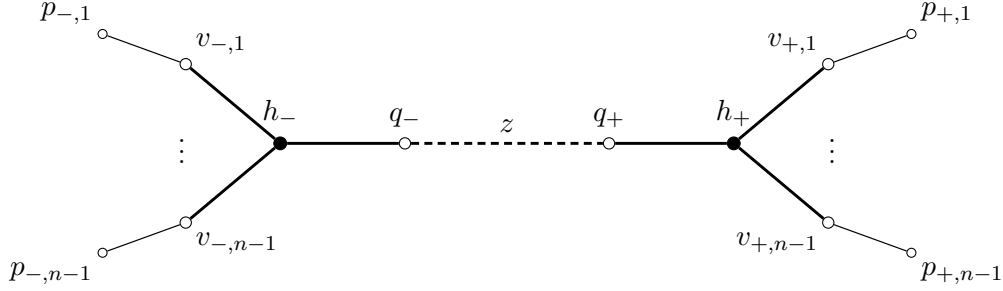
\begin{figure}[ht]
\centering
\begin{tikzpicture}[
  hub/.style={circle,draw,fill=black,inner sep=1.7pt},
  vertex/.style={circle,draw,fill=white,inner sep=1.5pt},
  leaf/.style={circle,draw,fill=white,inner sep=1.2pt},
  core/.style={line width=1.1pt},
  vanishing/.style={line width=0.45pt},
  cutting/.style={densely dashed,line width=1.1pt}]
\node[hub,label=above:$h_-$] (hm) at (-3,0) {};
\node[vertex,label=above:$q_-$] (qm) at (-1.35,0) {};
\node[vertex,label=above:$q_+$] (qp) at (1.35,0) {};
\node[hub,label=above:$h_+$] (hp) at (3,0) {};

\node[vertex] (vm1) at (-4.25,1.05) {};
\node[leaf]   (pm1) at (-5.35,1.45) {};
\node at (-4.3,0) {$\vdots$};
\node[vertex] (vmn) at (-4.25,-1.05) {};
\node[leaf]   (pmn) at (-5.35,-1.45) {};

\node[vertex] (vp1) at (4.25,1.05) {};
\node[leaf]   (pp1) at (5.35,1.45) {};
\node at (4.3,0) {$\vdots$};
\node[vertex] (vpn) at (4.25,-1.05) {};
\node[leaf]   (ppn) at (5.35,-1.45) {};

\draw[core] (hm)--(qm);
\draw[core] (hp)--(qp);
\draw[cutting] (qm)--node[above] {$z$} (qp);
\draw[core] (hm)--(vm1);
\draw[core] (hm)--(vmn);
\draw[core] (hp)--(vp1);
\draw[core] (hp)--(vpn);
\draw[vanishing] (vm1)--(pm1);
\draw[vanishing] (vmn)--(pmn);
\draw[vanishing] (vp1)--(pp1);
\draw[vanishing] (vpn)--(ppn);

\node[above right] at (vm1) {$v_{-,1}$};
\node[below right] at (vmn) {$v_{-,n-1}$};
\node[above left] at (vp1) {$v_{+,1}$};
\node[below left] at (vpn) {$v_{+,n-1}$};
\node[above left] at (pm1) {$p_{-,1}$};
\node[below left] at (pmn) {$p_{-,n-1}$};
\node[above right] at (pp1) {$p_{+,1}$};
\node[below right] at (ppn) {$p_{+,n-1}$};
\end{tikzpicture}
\caption{The tree in Example~\ref{ex:cuttingedge}.  The thick edges form
$C_-\cup C_+$ and have positive normalized limits.  The dashed internal edge
$z$ and the thin leaf edges have zero normalized limits.}
\label{fig:cuttingedge}
\end{figure}

Choose positive numbers $a_-,a_+,z_0$ and $b_e$, $e\in L_-\cup L_+$, and
prescribe the unnormalized initial metric by
\[
 \widetilde w_e(0)=a_\sigma\quad(e\in C_\sigma),
 \qquad
 \widetilde w_z(0)=z_0,
 \qquad
 \widetilde w_e(0)=b_e\quad(e\in L_-\cup L_+).
\]
Put $\gamma=(n-2)/n$.  If the edges in $C_\sigma$ have a common weight, then
their curvatures are equal.  Uniqueness of the flow therefore shows that these
edges retain a common weight.
Since $D_{h_\sigma}=n/\widetilde w_e$ for $e\in C_\sigma$, the curvature
formula gives, for all $t\ge0$,
\[
 \kappa_e=-\gamma\quad(e\in C_-\cup C_+),
 \qquad
 \kappa_z=0,
 \qquad
 \kappa_e=1\quad(e\in L_-\cup L_+).
\]
The flow equations give
\[
 \widetilde w_e(t)=a_\sigma e^{\gamma t}\quad(e\in C_\sigma),
 \qquad
 \widetilde w_z(t)=z_0,
 \qquad
 \widetilde w_e(t)=b_e e^{-t}\quad(e\in L_-\cup L_+).
\]
Writing $w=\widetilde w/\sum_f\widetilde w_f$, we obtain
\[
 \lim_{t\to\infty}w_e(t)
 =\frac{a_\sigma}{n(a_-+a_+)}>0\quad(e\in C_\sigma),
 \qquad
 \lim_{t\to\infty}w_z(t)=0,
 \qquad
 \lim_{t\to\infty}w_e(t)=0\quad(e\in L_-\cup L_+).
\]
Here $c_*=-\gamma$, and the decay rates are
\[
 \lim_{t\to\infty}\frac1t\log w_z(t)=-\gamma=c_*-c_z,
 \qquad
 \lim_{t\to\infty}\frac1t\log w_e(t)=-(1+\gamma)=c_*-c_e
 \quad(e\in L_-\cup L_+),
\]
in agreement with Theorem~\ref{thm:kappalimit}.  Removing $z$ separates the two
positive sets $C_-$ and $C_+$.  Hence $z$ is an internal cutting edge whose
normalized weight tends to zero.
\end{example}

This example also shows why the limiting curvature equations need not have a unique
solution.  For every $0<\theta<1$, define
\[
 p_e^{\theta}=\frac{\theta}{n}\quad(e\in C_-),
 \qquad
 p_e^{\theta}=\frac{1-\theta}{n}\quad(e\in C_+).
\]
Then $p^\theta\in\mathcal S_{C_-\cup C_+}(-\gamma)$.  The equations determine
the equal weights inside each $C_\sigma$, but not the proportion of the total mass
carried by the two sets.  For the trajectory above this proportion is fixed by the
initial metric: $\theta=a_-/(a_-+a_+)$.  Different initial metrics on
the same tree can have different normalized limits.  This does not contradict
convergence for a fixed initial metric: once the initial metric is fixed, so is the
value of $\theta$ selected by this trajectory.

The preceding results leave the following question open.

\begin{problem}\label{prob:weightconvergence}
Does the normalized inverse-weight Ricci flow converge for every positive initial
metric on every finite tree?
\end{problem}

By Theorem~\ref{thm:trajectorycomparison}, for a fixed tree it is enough to prove
convergence for one positive initial metric.  The convergence of the edge
curvatures does not settle the question.  Indeed, for $e,f\in E_*$,
\[
 \frac{d}{dt}\log\frac{w_e(t)}{w_f(t)}
 =\kappa_f(t)-\kappa_e(t)\longrightarrow0,
\]
but this does not imply convergence of the time integral, and hence does not by
itself imply convergence of the weight ratio.

\par\medskip
\noindent\textbf{Acknowledgements.}
S. Bai is supported by NSFC grant no.~12301434. B. Hua is supported by NSFC grant
no.~12371056.

\par\medskip
\noindent\textbf{Statement of AI Use.}
 The authors formulated the problem and the mathematical results,
 independently verified all arguments and references, and made all final decisions concerning the
 manuscript. GPT-5.6 Sol was used as a supporting tool for English-language and
 LATEX editing, manuscript organization, and discussion of the presentation of a proof argument,
 as well as for preliminary checking of calculations. The authors take full responsibility for the
 content.

\bibliographystyle{amsplain}
\bibliography{references,references_additions}

\end{document}